\documentclass[11pt,reqno,tbtags]{amsart}

\usepackage{atbegshi}
\newcommand{\handlethispage}{}
\AtBeginShipout{\handlethispage}
\newtheorem{xthm}{Theorem}
\newtheorem{xlem}[xthm]{Lemma} 
\newtheorem{xprop}[xthm]{Proposition}
\newtheorem{xcor}[xthm]{Corollary}
\newtheorem{xdfn}[xthm]{Definition}
\newtheorem{xfact}[xthm]{Fact}
\newif\ifonly 
\ifonly
\usepackage{answers}
\Newassociation{lem}{mtlem}{mtfile}
\Newassociation{thm}{mtthm}{mtfile}
\Newassociation{prop}{mtprop}{mtfile}
\Newassociation{cor}{mtcor}{mtfile}
\Newassociation{dfn}{mtdfn}{mtfile}
\Newassociation{fact}{mtfact}{mtfile}

\Opensolutionfile{mtfile}
\AtBeginDocument{%
\let\handlethispage\AtBeginShipoutDiscard}
\AtEndDocument{%
\clearpage\pagenumbering{arabic}
\let\handlethispage\relax
\Closesolutionfile{mtfile}
\Readsolutionfile{mtfile}}
\else
\newenvironment{thm}{\begin{xthm}}{\end{xthm}}
\newenvironment{lem}{\begin{xlem}}{\end{xlem}}
\newenvironment{prop}{\begin{xprop}}{\end{xprop}}
\newenvironment{cor}{\begin{xcor}}{\end{xcor}}
\newenvironment{dfn}{\begin{xdfn}}{\end{xdfn}}

\fi

\usepackage{amsthm,amssymb,amsfonts,amsmath}
\usepackage{amscd} 
\usepackage{mathrsfs}
\usepackage[numbers, sort&compress]{natbib}
\usepackage{subfigure, graphicx}
\usepackage{vmargin}
\usepackage{setspace}
\usepackage{paralist}
\usepackage[pagebackref=true]{hyperref}
\usepackage{color}
\usepackage[normalem]{ulem}
\usepackage{stmaryrd} 
\usepackage[refpage,noprefix,intoc]{nomencl} 

\providecommand{\R}{}
\providecommand{\Z}{}
\providecommand{\N}{}

\renewcommand{\R}{\mathbb{R}}
\renewcommand{\Z}{\mathbb{Z}}
\renewcommand{\N}{{\mathbb N}}

\newcommand{\E}[1]{{\mathbf E}\left[#1\right]}
\newcommand{\e}{{\mathbf E}}
\newcommand{\V}[1]{{\mathbf{Var}}\left\{#1\right\}}

\newcommand{\p}[1]{{\mathbf P}\left\{#1\right\}}
\newcommand{\psub}[2]{{\mathbf P}_{#1}\left\{#2\right\}}

\newcommand{\I}[1]{{\mathbf 1}_{[#1]}}

\newcommand{\esub}[1]{{\mathbf E_{#1}}}
\newcommand{\pcond}[3]{\mathbf{P}_{#3}\!\left(#1\;\middle\vert\;#2\right)}

 \newcommand{\bag}{\begin{align}}
\newcommand{\bags}{\begin{align*}}
\newcommand{\eag}{\end{align*}}
\newcommand{\eags}{\end{align*}}

\newcommand\cU{{\mathcal U}}

\newcommand{\pran}[1]{\left(#1\right)}

\providecommand{\eps}{}
\renewcommand{\eps}{\epsilon}
\providecommand{\ora}[1]{}
\renewcommand{\ora}[1]{\overrightarrow{#1}}

\hypersetup{
    bookmarks=true,         
    unicode=false,          
    pdftoolbar=true,        
    pdfmenubar=true,        
    pdffitwindow=true,      
    pdftitle={My title},    
    pdfauthor={Author},     
    pdfsubject={Subject},   
    pdfnewwindow=true,      
    pdfkeywords={keywords}, 
    colorlinks=true,       
    linkcolor=blue,          
    citecolor=blue,        
    filecolor=blue,      
    urlcolor=blue           
}

\newcommand\urladdrx[1]{{\urladdr{\def~{{\tiny$\sim$}}#1}}}

\newcommand{\GW}{\mathrm{GW}}

\usepackage[usenames,dvipsnames]{xcolor}

\newcommand{\h}{\mathop{\mathrm{ht}}}
\newcommand{\w}{\mathop{\mathrm{wid}}}
\newcommand{\gw}{\mathop{\mathrm{GW}}}

\numberwithin{equation}{section}
\numberwithin{xthm}{section}
\newcommand{\rb}{\ensuremath{\mathrm{b}}}

\begin{document}

\title{Most trees are short and fat} 
\author{Louigi Addario-Berry}
\address{Department of Mathematics and Statistics, McGill University, Montr\'eal, Canada}
\email{louigi.addario@mcgill.ca}
\date{March 30, 2017} 
\urladdrx{http://www.problab.ca/louigi/}

\subjclass[2010]{60J80,60G50,60E15,05C05} 

\begin{abstract} 
This work proves new probability bounds relating to the height, width, and size of Galton-Watson trees. 
For example, if $T$ is any Galton-Watson tree, 
and $H$, $W$, and $|T|$ are the height, width, and size of $T$, respectively, then $H/W$ has sub-exponential tails and $H/|T|^{1/2}$ has sub-Gaussian tails. 

Although our methods apply without any assumptions on the offspring distribution, when information is provided about the distribution the method can be adapted accordingly, and always seems to yield tight bounds.
\end{abstract}

\maketitle

\begin{quote}
``Il faudrait souhaiter que \ldots les probabilistes de la jeune g\'en\'eration, soit par complaisance pour les m\'ethodes purement analytiques, soit gr\^ace \`a l'engouement pour la puissance - d'ailleurs parfaitement justifi\'ee - des m\'ethodes li\'ees aux distributions dans les espaces fonctionnels, n'oublient point les m\'ethodes directes.
''\\
A.N. Kolmogorov, {\em Sur les propri\'et\'es des fonctions de concentrations de M. P. L\'evy, \cite{kolmogorov58}.}
\end{quote}
\section{Introduction}\label{sec:intro} 
Given a rooted tree $T$, write $T_n$ for the set of nodes of $T$ at distance $n$ from the root. The {\em width} of $T$ is $\w(T) = \sup_{n \ge 0} |T_n|$; the height of $T$ is $\h(T) = \sup\{n \ge 0: |T_n| \ne 0\}$. The {\em volume} of $T$ is its number of nodes, denoted $|T|$. We write $u \in T$ to mean that $u$ is a node of $T$. Note that if either $\w(T)=\infty$ or $\h(T)=\infty$ then $|T|=\infty$. 

Next, given a probability distribution $\mu$ supported on the the non-negative integers $\N$, write  
$T\sim \mathrm{GW}(\mu)$ if $T$ is a Galton-Watson tree with offspring distribution $\mu$. In other words, each node of $T$ gives birth to a random, $\mu$-distributed number of offspring, and the number of offspring is independent for distinct nodes. 

This work presents new probabilistic relations between the height, width, and volume of general Galton-Watson trees. Here is the most general result. 
\begin{thm}\label{thm:general}
There exists an absolute constant $C > 0$ such that the following holds. Fix a probability measure $\mu$ with support $\N$, and let $T\sim \mathrm{GW}(\mu)$. 
Then for all $x \ge 1$, 
\[
\p{\h(T) > C x \frac{\w(T)}{1-\mu(1)}} \le e^{-x}\, \, \mbox{and} \, \,
\p{\h(T) > C x \Big(\frac{|T|}{1-\mu(1)}\Big)^{1/2}} \le e^{-x^2}\, .
\]
\end{thm}
In the above theorem we write $\mu(1)$ instead of $\mu(\{1\})$. We use similar shorthand in several other places. Also, we say $\mu$ is {\em subcritical} if $\sum_i i\mu(i) < 1$, {\em critical} if $\sum_i i\mu(i)=1$, and {\em supercritical} if $\sum_i i\mu(i)>1$; in the third case, the sum may be infinite.

The method we introduce is rather general, in that when additional information is provided about the tail behaviour of $\mu$, the preceding bounds can be correspondingly strengthened. For example, we have the following two theorems. 
\begin{thm}\label{thm:fixed_var}
There exists an absolute constant $C > 0$ such that the following holds. Fix a critical or subcritical probability measure $\mu$ with support $\N$ and with $\sum_i i(i-1)\mu(i)=v \in (0,\infty)$, and let $T\sim \mathrm{GW}(\mu)$. 
Then there exists $x_0$ depending on $\mu$ such that for all $x \ge x_0$
\[
\p{\h(T) \ge Cx \w(T)} \le e^{-vx}\, \, \mbox{and} \, \,
\p{\h(T) \ge Cx |T|^{1/2}} \le e^{-vx^2}\, .
\]
Moreover, if $\mu$ is critical then we may choose $x_0$ so that for all $x \ge x_0$  and all $n \ge 1$, 
\[
\p{\h(T) \ge Cx |T|^{1/2}~|~|T| \ge n} \le e^{-vx^2} \, .
\]
\end{thm}
Theorem~\ref{thm:fixed_var} is related to a result of \cite{addario-berry2013}, which proves sub-Gaussian tail bounds for $h(T)/|T|^{1/2}$ when $T$ has finite variance. The results of \cite{addario-berry2013} apply to Galton-Watson trees conditioned to have a fixed size, and in this sense are stronger than Theorem~\ref{thm:fixed_var}. However, in \cite{addario-berry2013} the bounds become weaker as the variance $v$ increases, whereas the bounds of Theorem~\ref{thm:fixed_var} exhibit the correct dependence on the variance. 
\begin{thm}\label{thm:stable}
There exists an absolute constant $C > 0$ such that the following holds. Fix a critical or subcritical probability measure $\mu$ with support $\N$, and let $T\sim \mathrm{GW}(\mu)$. 
Suppose that there exists $\alpha \in (1,2]$ and $M > 0$ such that $\mu([i,\infty)) \le Mi^{-\alpha}$ for all $i \ge 1$. 
Then for all $x \ge 1$, 
\[
\p{\h(T) \ge  x \left(\frac{C M}{\alpha-1}\right)^{\alpha} \w(T)^{\alpha-1}}\le e^{-x} \,\, \mbox{and} \, \,
\p{\h(T) \ge C M x \frac{|T|^{(\alpha-1)/\alpha}}{(\alpha-1)}} \le e^{-x^\alpha}\, .
\]
\end{thm}

To understand the requirement that $\alpha \in (1,2]$ in Theorem~\ref{thm:stable}, note that if $\mu([i,\infty)) \ge c i^{-\alpha}$ for some $\alpha \le 1$ then $\sum_i i \mu([i,\infty)) = \infty$ so $\mu$ is not subcritical. On the other hand, if $\mu$ is in the domain of attraction of an $\alpha$-stable distribution (with $\alpha \in (1,2]$) then both bounds have optimal dependence on $\alpha$; see \cite{MR1989629,kortchemski17,duquesne17}. 
On the other hand, if $\mu([i,\infty)) \le M i^{-\alpha}$ for some $\alpha > 2$ then $\mu$ is again in the domain of attraction of the $2$-stable (Gaussian) distribution, in which case  Theorems~\ref{thm:general} and~\ref{thm:fixed_var} apply. 

Theorem~\ref{thm:general} is tight in general; we explain why shortly. The next theorem shows that when $\mu$ has infinite variance, Theorem~\ref{thm:general} is not tight for large trees. 
\begin{thm}\label{thm:inf_var}
Fix a critical or subcritical probability measure $\mu$ with support $\N$, and let $T\sim \mathrm{GW}(\mu)$. If $\sum_i i^2 \mu(i) = \infty$ 
then for any $\eps > 0$ there exists $n_0$ depending on $\eps$ and $\mu$ such that for all 
$x > 0$ and all $n \ge x^2 n_0$, 
\[
\p{\h(T) \ge x|T|^{1/2}, \sigma \ge n} \le \frac{x}{n^{1/2}} e^{-x^2/\eps}\, \, \, \mbox{and}\,\, 
\p{\h(T) \ge x\w(T), \sigma \ge s} \le \frac{x}{n^{1/2}} e^{-x/\eps}\, .
\]
\end{thm}
The following example limits the degree to which 
Theorem~\ref{thm:inf_var} can be strengthened. Suppose that $\mu(0)=p >0$ and fix $k > 0$ with $\mu(k)=q > 0$. Let $E_h=E_h(k)$ be the event that $\h(T)=h$, that $T$ contains exactly $h$ nodes with $k$ children and that all other nodes are leaves. 
$\p{E_h} \ge q^h p^{hk} = e^{-ch}$, where $c=c(\mu)>0$. 
When $E_h$ occurs we have $|T|= kh+1$ and $\h(T)= |T| \cdot h/(kh+1)$, so with $x = h/(kh+1)^{1/2}$, we obtain
\[
\p{\h(T) \ge x |T|^{1/2},\sigma \ge kh+1} \ge \p{E} \ge 
e^{-ch} \ge e^{-c x^2}> k e^{-ckx^2}\, .
\]
This implies that the dependence of $n_0$ on $\eps$ in Theorem~\ref{thm:inf_var} can not be removed, and that the quadratic relation between the allowed values of $n$ and $x$ is essentially tight.

If $\mu(1)=q>0$ then the same example shows that the bounds of Theorem~\ref{thm:general} are tight. In this case $E_h=E_h(1)$ is the event that $T$ is a path of length $h$, and with $x=(h(1-q))^{1/2}$, 
we have 
\[
\p{\h(T) \ge x \left(\frac{|T|}{1-q}\right)^{1/2}}
\ge 
\p{E_h} = pq^h = pe^{-h\log(1/q)} \ge pe^{-x^2/q}. 
\]

The remainder of the introduction is structured as follows. Section~\ref{sec:prelim} provides a few basic definitions, and states a well-known identity in law between the breadth-first queue process of a Galton-Watson tree and a corresponding random walk stopped on its first visit to zero. In Section~\ref{sec:proof} we introduce a key bound, used in the proofs of all the above theorems, which shows that the height of a tree is bounded, up to a scaling factor, by the ``inverse harmonic average queue length''. We then provide a brief overview of the overall proof strategy and of the structure of the remainder of the paper. 

\subsection{Preliminaries}\label{sec:prelim}
For a rooted tree $T$, write $r(T)$ for the root of $T$. For $v \in T$, write $c_T(v)$ for the number of children of $v$ in $T$, and write $p(v)=p_T(v)$ for the parent of $v$ in $T$, with the convention that $p(r(T))=r(T)$. Also, write $|v|$ for the distance from $v$ to $r(T)$, and recall that $T_k=\{v \in T: |v|=k\}$ is the set of nodes of $T$ at depth $k$.

The Ulam-Harris tree is the infinite tree $\cU$ with root $\emptyset$ and non-root nodes indexed by finite sequences of positive integers. A node $v=v_1\dots v_k$ at depth $k$ has parent $v_1\dots v_{k-1}$ and children $\{vj,j \ge 1\}=\{v_1\dots v_kj,j \ge 1\}$, where $v_1\dots v_{k-1}=\emptyset$ by convention if $k=1$. For nodes $v=v_1 \dots v_k$ and $w=w_1\dots w_l$ of $\cU$, write $v \prec_\cU w$ if either $k < l$, or $k=l$ and $v$ precedes $w$ lexicographically. We refer to $\prec_\cU$ as the breadth-first ordering of the nodes of $\cU$. 

Any countable rooted plane tree $T$ may be viewed as a subtree of $\cU$ as follows. First identify the root of $T$ with the root $\emptyset$ of $\cU$. Next, recursively, if $x \in T$ is identified with $v=v_1\dots v_k \in \cU$ then identify the $j$th child of $x$ with $vj$, for each $j \le c_T(u)$. With this identification, the ordering $\prec_\cU$ induces a total order $\prec_T$ of the nodes of $T$, which we call the breadth-first ordering of the nodes of $T$. If $T$ is locally finite then we enumerate the nodes of $T$ in breadth-first order as $v(T)=\{v_{i+1}(T),0 \le i < |T|\}$; writing the indices this way allows us to treat the cases $|T| =\infty$ and $|T|<\infty$ simultaneously. 

For $v \in v(T)$, if $v=v_i(T)$ then let 
$S(v) = S_T(v) = 1 + \sum_{j < i} (c_T(v_j(T))-1)$. 
The quantity $S(v)$ is the {\em breadth-first queue length} just before $v$ is explored in a lexicographic breadth-first search of $T$. 
It is not hard to see that for all $v \in T$, 
\[
S(v) = \#\{w \in T: p(w) \preceq_T v \preceq_T w\}\, ;
\]
recall that $p(r(T))=r(T)$ by convention. 
\begin{prop}\label{bfq_equiv}
Let $\mu$ be a probability measure on $\N$, 
and define a probability measure $\nu$ on $\{-1\} \cup \N$ by taking $\nu(i)=\mu(i+1)$ for $i \in \{-1\} \cup \N$. 
Let $(S_i,i \ge 0)$ be a random walk with initial position $S_0=1$ and jump distribution $\nu$, and let $\sigma = \inf\{n: S_n=0\}$. If $T \sim \gw(\mu)$ then 
 $(S(v_{i+1}),0 \le i < |T|)$ and $(S_i,0 \le i < \sigma)$ are identically distributed. 
\end{prop}

For a random walk $S=(S_n,n \ge 0)$ we write $\psub{u}{\cdot}$ for the probability measure under which the random walk has initial position $S_0=u$. Our default notation for the steps of $S$ is $(X_i,i \ge 1)$, so that $S_n-S_{n-1}=X_n$ for all $n \ge 1$.

\subsection{Proof overview}\label{sec:proof}
The following chain of identities, 
\[
\h(T) = \sum_{k=1}^{\h(T)} 1 = \sum_{k=1}^{\h(T)} \sum_{v \in T_k} \frac{1}{|T_k|}\, ,
\]
may seem too trivial to be useful, and it is. However, it contains an idea which, after a little massaging, becomes 
quite powerful, and indeed undergirds the rest of the proof. The following lemma and proposition show that replacing 
$1/|T_k|$ by $3/S(v)$ within the final sum yields an upper bound on $\h(T)$, and allows us to bound the height by studying the 
breadth-first queue process. 
\begin{lem}\label{lem:two_sums}
For any sequence of positive integers $(n_k,0 \le k \le h+1)$ with $n_0=n_{h+1}=1$, it holds that 
\begin{equation}\label{eq:two_sums}
\sum_{k\in [0,h]: n_k \le n_{k+1}} \frac{n_k}{n_k+n_{k+1}} + \sum_{k\in [0,h]:n_k > n_{k+1}} \log\pran{\frac{n_k+n_{k+1}}{n_{k+1}}} \ge \frac{h}{3}\, .
\end{equation}
\end{lem}
\begin{proof}
In what follows, $k$ is always assumed to be an integer from the interval $[0,h]$; 
so, for example, $\{k: n_k > n_{k+1}\}$ is shorthand for the set $\{k \in \{0,1,\ldots,h\}: n_k > n_{k+1}\}$. 

First note the easy bound 
\[
\sum_{k:n_k > n_{k+1}} \log\pran{\frac{n_k+n_{k+1}}{n_{k+1}}} \ge (\log 2) \cdot \#\{k: n_k > n_{k+1}\}. 
\]
A slightly more complicated lower bound on the same sum follows from the identity
$\prod_{0 \le k < h} \frac{n_k}{n_{k+1}} = \frac{n_0}{n_{h+1}}=1$, which implies that 
\begin{align*}
\sum_{k:n_k > n_{k+1}} \log\pran{\frac{n_k+n_{k+1}}{n_{k+1}}} 
	& > \sum_{k:n_k > n_{k+1}} \log\pran{\frac{n_k}{n_{k+1}}} \\
	& = \log\pran{\prod_{k:n_k > n_{k+1}}\frac{n_k}{n_{k+1}}}\\
	& = \log\pran{\prod_{k:n_{k+1} \ge n_k}\frac{n_{k+1}}{n_k}} \\
	& \ge \log\pran{\prod_{k:n_{k+1} \ge 2n_k}\frac{n_{k+1}}{n_k}} \\
	& \ge (\log 2)\cdot \#\{k: n_{k+1} \ge 2n_k\}\, .
\end{align*}
Finally, 
\[
\sum_{k: n_{k+1} \in [n_k,2n_k)} \frac{n_k}{n_k+n_{k+1}} \ge \frac{1}{3} \cdot \#\{k: n_{k+1} \in [n_k,2n_k)\}. 
\]
It follows that for any $p \in [0,1]$, the left-hand side of (\ref{eq:two_sums}) is at least 
\[
\frac{1}{3} \cdot \#\{k: n_{k+1} \in [n_k,2n_k)\} + (p\cdot \#\{k: n_{k+1} < n_k\} + 
(1-p)\cdot \#\{k: n_{k+1} \ge 2n_k\})\cdot \log 2. 
\]
Since $\log 2 > 2/3$, setting $p=1/2$ wraps things up. 
\end{proof}
\begin{prop}\label{prop:ht_bd}
For any tree $T$, it holds that $\h(T) \le 3\cdot \sum_{v \in T} \frac{1}{S(v)}$. 
\end{prop}
\begin{proof}
``Shifting time by one'' makes the numbers work out a little more smoothly. For $1 \le i < |T|$ let $M(v_i) = S(v_{i+1})$; 
this is the length of breadth-first queue just {\em after} $v_i$ is explored. Since $S(v_1)=1$, setting $M(v_{|T|})=1$ yields 
\[
\sum_{v \in T} \frac{1}{S(v)} = \sum_{v \in T} \frac{1}{M(v)}.
\]

For $0 \le k \le \h(T)$ let $n_k = |T_k|$, and let $n_{\h(T)+1}=1$. 
Now fix an integer $k \in [0, \h(T)]$. In a breadth-first exploration of $T$, for all $v \in |T_k|$, 
just after $v$ is explored the queue consists of nodes from $T_k \cup T_{k+1}$. 
This implies that 
\[
\sum_{v \in T_k} \frac{1}{|M(v)|} \ge |T_k| \cdot \frac{1}{|T_k\cup T_{k+1}|} = \frac{n_k}{n_k+n_{k+1}}\, .
\]

Next, with $m = \sum_{0 \le j < k} |T_j|$, the nodes of $T_k$ in breadth-first order are 
$v_{m+1},\ldots,v_{m+n_k}$. Just after $v_{m+i}$ is explored the elements of queue form a subset of
$\{v_{m+i+1},\ldots,v_{m+n_k}\} \cup T_{k+1}$, so 
\[
\sum_{v \in T_k} \frac{1}{|M(v)|} \ge \sum_{i=0}^{n_k} \frac{1}{n_k+n_{k+1}-i} \ge \log\pran{\frac{n_k+n_{k+1}}{n_{k+1}}}\, .
\]
Together, these inequalities yield 
\[
\sum_{v \in T} \frac{1}{S(v)} = \sum_{v \in T} \frac{1}{|M(v)|} \ge \sum_{k=0}^{\h(T)} \max\pran{\frac{n_k}{n_k+n_{k+1}},\log\pran{\frac{n_k+n_{k+1}}{n_{k+1}}}}\, ,
\]
and the proposition follows from Lemma~\ref{lem:two_sums}.
\end{proof}

Note that the width is also easily bounded using the breadth-first queue. First, as noted in the preceding proof, if $v \in T_k$ then the nodes counted by $S(v)$ form a subset of $T_k \cup T_{k+1}$. Next, if $v_k$ is the {\em first} node of $T_k$ with respect to the breadth-first order $\prec_{T}$, then $S(v_k) = |T_k|$. Together, these facts immediately imply that 
\begin{equation}\label{walk_width}
\frac{1}{2}\max_{v\in T} S(v) \le \w(T) \le \max_{v\in T} S(v)\, .
\end{equation} 

When $T\sim \gw(\mu)$, we have $(S(v_{i+1}),0\le i < |T|) \sim (S_i,0 \le i < \sigma)$, where 
$(S_i,i \ge 0)$ is a simple random walk with $S_0=1$ whose steps $(X_i,i \ge 1)$ satisfy $X_i+1\sim \mu$, 
and where $\sigma= \inf\{t: S_t=0\}$. 

For $0 \le t \le \sigma$ we write 
\[
H(t) = \sum_{i < t} \frac{1}{S_i}. 
\] 
Proposition~\ref{prop:ht_bd} and the bounds of (\ref{walk_width}) reduce 
our task to that of understanding the relations between $\sigma$ and the quantities $H(\sigma)$ and $\max_{0 \le i < \sigma} S_i$. 

At the heart of our argument is a multi-scale decomposition of the random walk path $(S_i,0 \le i \le \sigma)$. 
When $S_t$ takes values at a given scale -- when $S_t$ has order about $x$, say -- then for as long as the random walk continues to have order about $x$, the partial sums $\sum_{0 \le i \le t} \frac{1}{S_i}$ increase at rate about $1/x$. In particular, it takes $\Theta(x)$ steps for the sequence of partial sums to 
increase by $1$. 

To convert this into a good bound requires controlling how much time the random walk spends at each scale. 
This in turn impels us to track the times  the random walk ``changes scale'' and how many times it may revisit a given scale before 
visiting the origin, and to produce bounds on how long the random walk might spend at a given scale. For the latter, we make crucial use of universal bounds on the {\em concentration} function of a sum of IID random variables. The bounds may be sharpened if further information is provided about the jump distribution, and such sharpenings seem to generally yield optimal bounds for the heights of the corresponding Galton-Watson trees. 

In Section~\ref{sec:exit}, we introduce a general bound on the concentration function of a random walk, and deduce several corollaries which hold under specific hypotheses about the step distribution. 
In Section~\ref{sec:scales}, we formally define the multi-scale decomposition which is at the core of our proofs, and prove results controlling the amount of time spent by a random walk at a fixed scale. In Section~\ref{key_results}, these results are applied to prove probabilistic bounds relating $H(\sigma)$ and $\sigma$ under various assumptions on the step size; from these, the theorems stated in the introduction are straightforwardly derived, in Section~\ref{sec:proofs}. Finally, Section~\ref{sec:conc} contains some questions and references to the literature.

\section{Exit times from intervals and the concentration function}\label{sec:exit}
Following \citet{MR0301786}, for a random variable $Z$, and a constant $L > 0$, write 
\[
Q(Z,L) = \sup_{x \in \R}\p{Z \in [x,x+L)}. 
\]
This function was introduced by Paul L\'evy \citep[Section 16]{levy}, who called it the concentration function of $Z$. 
We will use the following bounds. 
\begin{thm}[\cite{MR0301786}]\label{thm:kesten}
There exists an absolute constant $C$ such that the following holds. 
Let $(X_i,i\ge 1)$ be iid copies of a random variable $X$, fix $n \in \N$ and let $S=\sum_{i=1}^n X_i$. Then for all $L >0$, 
\[
Q(S,L) \le \frac{CL}{(n\E{(X_1-X_2)^2 \I{|X_1-X_2| \le L}})^{1/2}}\, ,
\]
and 
\[
Q(S,L) \le \frac{C \cdot L}{n^{1/2}(1-Q(X,1/2))^{1/2}}\, .
\]

\end{thm}
The next two corollaries are straightforward applications of Theorem~\ref{thm:kesten}. 
Recall that for a random walk $(S_n,n \ge 0)$, we write $\mathbb{P}_u$ for the probability measure under which the walk has initial position $u \in \R$. 
\begin{cor}\label{cor:infvar}
Let $(X_i,i \ge 1)$ be iid copies of a random variable $X$, and let $(S_n,n \ge 0)$ be a random walk with steps $(X_i,i \ge 1)$. Suppose that $\E{X^2}=\infty$. For $K \ge 1$, let $\sigma(K) = \inf\{n: S_n \not \in (0,2^K)\}$. Then for any $\eps > 0$, 
\[
\lim_{K \to \infty} \sup_{x \in (0,2^K)} \psub{x}{\sigma(K) \ge \eps \cdot 4^K} =0\, .
\]
\end{cor}
\begin{proof}
Since $\E{X_1^2}=\infty$, we also have $\E{(X_1-X_2)^2}=\infty$. Fixing $\eps >0$, we may thus find $L$ such that $\E{(X_1-X_2)^2 \I{|X_1-X_2| \le L}} \ge \eps^{-2}$. It follows by Theorem~\ref{thm:kesten} that for all $K$ with $2^K \ge L$, for any $s \in (0,2^K)$, 
\[
\psub{x}{\sigma(K) \ge \eps \cdot 4^K} 
\le \psub{x}{S_{\lceil\eps 4^K\rceil} \in (0,2^K)} \le 
Q(\lceil \eps 4^K\rceil,2^K) \le \frac{C 2^K}{(\eps 4^K)^{1/2} \cdot \eps^{-1}}=C\eps^{1/2}\, .
\]
The result follows. 
\end{proof}
\begin{cor}\label{cor:disperse}
There exists an absolute constant $C$ such that the following holds. 
Let $(X_i,i \ge 1)$ be iid copies of an integer random variable $X$ with $\p{X \ge -1}= 1$, and let $(S_n,n \ge 0)$ be a random walk with steps $(X_i,i \ge 1)$. Then for all $n \in \N$ and all $m \in \N$, 
\[
Q(S_n,2^m) \le C \cdot \frac{2^m}{(n\p{X=-1} \E{X^2\I{X \in [0,2^m)}})^{1/2}}\, .
\]
\end{cor}
\begin{proof}
Taking $L=2^m$ for positive integer $m$, we have 
\begin{align*}
\E{(X_1-X_2)^2\I{|X_1-X_2| \le 2^m}} & \ge 2\sum_{0 \le  i < m} 2^{2(i-1)}\p{X_1\in [2^i,2^{i+1}),X_2 \le 2^{i-1}} \\
				& \ge \frac{\p{X_2=-1}}{16} \cdot \sum_{0 \le i < m} 2^{2(i+1)} \p{X_1\in [2^i,2^{i+1})} \\
				& \ge \frac{\p{X=-1}}{16} \E{X^2 \I{X \in [0,2^m)}}, 
\end{align*}
the final inequality since if $X_1=0$ then $X_1^2=0$. 
The result now follows from the first bound of Theorem~\ref{thm:kesten}. 
\[
Q(S_n,2^{m}) \le \frac{16C \cdot 2^m}{n^{1/2} \p{X=-1} \cdot \E{X^2\I{X \in [0,2^m)}}}\, . 
\]
The result follows. 
\end{proof}

\begin{lem}\label{lem:nl_bound}
There exists an absolute constant $C$ such that the following holds. 
Let $(X_i,i \ge 1)$ be iid copies of an integer random variable $X$ with $\p{X \ge -1}= 1$ and with $\e{X} \le 0$, and let $(S_n,n \ge 0)$ be a random walk with steps $(X_i,i \ge 1)$. Fix $K \ge 1$ and let $\sigma(K) = \inf\{n: S_n \not \in (0,2^K)\}$. 
Then 
\[
\sup_{x \in (0,2^K)} 
\psub{x}{\sigma(K) \ge C \cdot \frac{4^K}{1-\p{X=0}}} \le \frac{1}{2}, 
\]
\end{lem}
\begin{proof}
Write $p_{\max} = \max(\p{X=-1},\p{X=0},\p{X \ge 1})$. 
By the assumptions of the lemma, $Q(X,1/2) \le p_{\max}$, so 
by the second bound in Theorem~\ref{thm:kesten}, for any $x \in [0,2^K)$ and any $t \in \N$ we have 
\begin{equation}\label{eq:kesten_app}
\psub{x}{\sigma(K) > t} \le \psub{x}{S_t \in (0,2^K)} \le Q(S,2^K) < \frac{C\cdot 2^K}{t^{1/2}(1-p_{\max})^{1/2}}. 
\end{equation}
If $p_{\max} = \p{X=0}$ then provided $t > C^2 \cdot 4^{K+1}/(1-\p{X=0})$ the above bound yields 
$\psub{x}{\sigma(K) \ge t} \le 1/2$. 

Using the result of the preceding paragraph, we may assume that $p_{\max} \ne \p{X=0}$. In this case, since $0 \ge \e{X} \ge \p{X \ge 1}-\p{X=-1}$ we must have $p_{\max} = \p{X=-1} \ge 1/3$. 
If $p_{\max} \le 3/4$ then the bound in (\ref{eq:kesten_app}) is at most $2C \cdot 2^K/(t^{1/2}(1-\p{X=0})^{1/2})$, 
so $\psub{x}{\sigma(K) \ge t} \le 1/2$ whenever $t >  C^2 \cdot 4^{K+2}/(1-\p{X=0})$, 
proving the lemma in this case. For the remainder of the proof we thus assume $\p{X=-1} \ge 3/4$. 

Note that if $i$ is such that $X_i \ge 2^K$ then either $S_{i-1} < 0$ or $S_i \ge 2^K$, so if $\max(X_i, 1 \le i \le t) \ge 2^K$ then $\sigma(K) \le t$. For all $x \in (0,2^K)$ we thus have 
\begin{align*}
\psub{x}{\sigma(K) > t} 	& = \psub{x}{\sigma(K) > t, \max(X_i, 1 \le i \le t)< 2^K} \\
					& \le \p{\sum_{i=1}^t X_i \I{X_i \le 2^K} > -2^K} 
\end{align*}

Suppose $\E{X \I{X \in [0,2^K)}} \le 1/2$. Then we have $\E{X \I{X < 2^K}} \le -1/4$, so defining
$\hat{S} = \sum_{i=1}^t X_i \I{X_i \le 2^K}$, by Chebyshev's inequality the latter probability is at most
\[
\p{\hat{S} \ge \e\hat{S} + t/4-2^K} \le \frac{\V{\hat{S}}}{(t/4-2^K)^2}\, .
\]
If $t > 8 \cdot 2^K$ then $(t/4-2^K) < t/8$, so this yields 
\[
\psub{x}{\sigma(K) \ge t} = \frac{64\V{X \I{X < 2^K}}}{t}\, .
\]
The random variable $X \I{X < 2^K}$ has support $[-1,2^K)$ and mean at most zero, which implies that  
$\V{X \I{X < 2^K}} < 2^K$. The preceding bound then gives 
\begin{equation}\label{eq:neg_mean}
\psub{x}{\sigma(K) > t} \le \frac{64 \cdot 2^K}{t}\, .
\end{equation}
proving the lemma in this case. 

Finally, suppose that $\E{X \I{X \in [0,2^K)}} > 1/2$. Then applying the conditional Jensen's inequality gives 
\begin{align*}
\E{X^2 \I{X \in [0,2^K)}} & = \E{X^2\left|X \in [0,2^K)\right.} \p{X \in [0,2^K)} \\
				& \ge \Big(\E{X\left|X \in [0,2^K)\right.}\Big)^2 \p{X \in [0,2^K)} \\ 
				& = \frac{\E{X \I{X \in [0,2^K)}}^2}{\p{X \in [0,2^K)}} \\
				& \ge 1\, ,
\end{align*}
the last inequality since $\p{X \in [0,2^K)} \le 1/4$. 
Corollary~\ref{cor:disperse} yields that
\[
Q(S_t,2^K) \le C \cdot \frac{2^K}{(t\p{X=-1})^{1/2}} \le \frac{C}{(3/4)^{1/2}} \frac{2^K}{t^{1/2}}. 
\]
For $t > (4C^2/3) \cdot 4^{K+1}$, we then have 
$\p{\sigma(K) \ge t} \le  Q(S_t,2^K) \le 1/2$, and the lemma follows. 
\end{proof}

\section{Decomposing a random walk into scales}\label{sec:scales}
Throughout Sections~\ref{sec:scales} and~\ref{key_results}, fix an integer random variable $X$ with $\e X \le 0$ and $\p{X=-1}=1$, and let $(S_t,t \ge 0)$ be a random walk with 
steps distributed as $X$. 
Let $\sigma = \min\{t: S_t \le 0\}$ be the first time the random walk visits the non-positive integers, and for $K \ge 1$ let $\sigma(K) = \min\{t: S_t \not \in (0,2^K)\}$. 
\subsection{Exit probabilities and upcrossings}
We begin with an easy fact about exit probabilities. 
\begin{lem}\label{lem:interval} 
Fix integers $a \le z < b$, and let $\tau = \inf\{t: S_t \not \in [a,b)\}$. 
Then $\psub{z}{S_{\tau} \ge b} \le (z+1-a)/(b+1-a)$. 
\end{lem}
\begin{proof}
It is easily seen that $\esub{z}{\tau} < \infty$. Since $(S_t,t \ge 0)$ is a submartingale, it follows that 
$\esub{z}{S_{\tau}} \le z$. Now note that since $S_{t+1} \ge S_t-1$ for all $t$, 
when $S_0=z \in [a,b)$ we have either $S_{\tau}=a-1$ or $S_{\tau} \ge b$. 
Writing $p=\psub{z}{S_{\tau} > b}$, it follows that 
\[
z \ge \esub{z}{S_{\tau}} = \esub{z}{S_{\tau}\I{S_{\tau}\ge b}}+ \esub{z}{S_{\tau}\I{S_{\tau}=a-1}} \ge bp + (a-1)(1-p). \qedhere
\]
\end{proof}
A random walk makes an {\em upcrossing} of an interval $[x,y) \subset \R$ each time it travels from a location below $x$ to a location above $y$. 
Let $\tau_0^- = \inf\{t: S_t < x\}$. Then, for $i \ge 0$, let $\tau_i^+ = \inf\{t > \tau_i^-: S_t \ge y\}$ and let $\tau_{i+1}^- = \inf\{t > \tau_i^+: S_t < x\}$. 
The random walk $S$ finishes its $i$'th upcrossing of $[x,y)$ at time $\tau_i^+$. Write $U(t)=U(t;[x,y)) = \max\{i: \tau_i^+ \le t\}$. 
The following result states that the number of upcrossings of an 
interval before the first visit to zero is stochastically dominated by an appropriate geometric random variable. 
\begin{prop}\label{prop:upcrossing}
Fix integers $0 < x < y$, and let $U(t;[x,y))$ be the number of upcrossings of $[x,y)$ by $S$ by time $t$. 
Also, let $\sigma = \min\{t: S_t \le 0\}$. Then for any positive integers $x$ and $k$, 
\[
\psub{x}{U(\sigma;[x,y)) \ge k} \le \pran{\frac{x-1}{y}}^k\, .
\]
\end{prop}
\begin{proof}
Write $U(\sigma)=U(\sigma;[x,y))$.  For $k \ge 1$, $U(\sigma) \ge k$ if and only if $\tau_k^+ < \sigma$.  
We use that for events $A \subset B \subset C$, it holds that $\pcond{A}{C}{} \le \pcond{A}{B}{}$. 
By the inclusions $\{\tau_{k}^+ < \sigma\} \subset \{\tau_{k}^- < \sigma\} \subset \{\tau_{k-1}^+ < \sigma\}$ 
we thus have 
\begin{align*}
\pcond{U(\sigma) \ge k}{U(\sigma)\ge k-1}{s} 	& = \pcond{\tau_{k}^+<\sigma}{\tau_{k-1}^+<\sigma}{s}\\
									&\le \pcond{\tau_{k}^+<\sigma}{\tau_{k}^-<\sigma}{s}\, .
\end{align*}
Next, observe that $S_{\tau_{k}^-} = x-1$. For a random walk starting from $x-1$, we have $\tau_0^+ < \sigma$ precisely if $x \ne 1$ and the random walk 
exits the interval $[1,y)$ in the positive direction. Using the strong Markov property, and Lemma~\ref{lem:interval} with 
$a=1$, $z=x-1$ and $b=y$, it follows that 
\[
\pcond{\tau_{k}^+<\sigma}{\tau_{k}^-<\sigma}{s} = \psub{x-1}{\tau_1^+ < \sigma} \le \frac{x-1}{y}. 
\]
By induction this yields $\psub{s}{U(\sigma) \ge k} \le (\frac{x-1}{y})^k \psub{s}{U(\sigma) \ge 0}$, from which the proposition is immediate. 
\end{proof}

\subsection{The scale of a random walk}\label{sec:scale}
We now describe a collection of stopping times that formalizes the notion of ``change of scale'' for the random walk. 
These scales are designed to be overlapping, the utility of which is that when the walk switches from one scale to another, smaller scale, it is reasonably unlikely 
to ever revisit a larger scale. 

In brief, a change of scale occurs when the walk leaves an interval of the form $(2^{i-1},2^{i+2}]$; just before this happens 
the random walk was at scale $i$. When the scale changes, the next scale, $j$, is chosen so that the random walk lies between $2^j$ and $2^{j+1}$; it now must exit 
the interval $(2^{j-1},2^{j+2}]$ to change scale, and so forth. 

\begin{dfn}
Let $\tau_0=0$ and $L_0=\sup\{\ell: S_{\tau_0} \ge 2^{\ell}\}$. 
Next, for $i \ge 0$, let 
\[
\tau_{i+1} = \min\left\{t \ge \tau_i: S_t \not \in [2^{L_i-1},2^{L_i+2})\right\},
\]
and let $L_{i+1} = \max\{\ell: S_{\tau_{i+1}} \ge 2^\ell\}$. 
\end{dfn}
The sequence $(\tau_i,i \ge 0)$ captures the times at which the random walk changes scale, and the sequence $(L_i,i \ge 0)$ contains the scales. 
For $i \ge 0$, if $\tau_{i+1} > \tau_i$ then $S_{\tau_{i+1}} \ge 2^{L_i+2}$ so $L_{i+1} \ge L_i+2$. Also, since the walk only makes negative steps of size $1$, 
if $S_{\tau_{i+1}} < S_{\tau_i}$ then $S_{\tau_{i+1}}=2^{L_i-1}-1$. 
Thus, if $\tau_{i+1} < \tau_{i}$ and $L_i\ge 2$ then $L_{i+1}=L_i-2$. Furthermore, $\sigma=\inf\{\tau_i: L_i = -\infty\}= \inf\{\tau_i: L_i \le 0\}$. 
In summary: when the scale changes, it either increases by at least two, or decreases by exactly two, or jumps to $-\infty$, the latter occurring 
at the hitting time of $0$. The right mental picture is to stop the walk at time $\sigma$. The definitions imply that once $L_i=-\infty$, it also holds that $\tau_j=\tau_i$ and $L_j=-\infty$ for all $j \ge i$. It is handy to also set $\tau_{\infty}=\sigma$. We take $\inf \emptyset = \infty$ by convention.
\begin{dfn}
Fix $\ell \ge 0$. Let $i(\ell,1) = \inf\{i: L_i=\ell\}$, and for $m >1$ let $i(\ell,m) = \inf\{j: j> i(\ell,m-1), L_j=\ell\}$. Then let $M(\ell) = \max\{m: \tau_{i(\ell,m)} < \sigma\}$. 
\end{dfn}
Note that $i(\ell,m)=\infty$ if there is no $j>i(\ell,m-1)$ with $L_j=\ell$. 
When $i(\ell,m)< \infty$, the random variable $\tau_{i(\ell,m)}$ is the $m$'th time the walk visits scale $\ell$, and $i(\ell,m)$ is the number of changes of scale before time $\tau_{i(\ell,m)}$. 
Also, $M(\ell)$ is the number of visits to scale $\ell$ before the walk hits $0$. 
\begin{prop}\label{prop:up_bd}
For all $\ell \ge 0$, $M(\ell) \le U(\sigma;[2^{\ell-1},2^{\ell}))+ U(\sigma;[2^{\ell+1},2^{\ell+2}))$. 
\end{prop}
\begin{proof}
First, for all $m \ge 1$ with $i(\ell,m) < \infty$ we have $2^\ell \le S(\tau_{i(\ell,m)}) < 2^{\ell+1}$. 
If $L_{i(\ell,m)+1} > \ell$ then $S(\tau_{i(\ell,m)+1}) \ge 2^{\ell+2}$ so at time 
$\tau_{i(\ell,m)+1}$ the random walk has just completed an upcrossing of $[2^{\ell+1},2^{\ell+2})$. 

If $L_{i(\ell,m)+1} < \ell$ and $\ell=\in \{0,1\}$ then $\tau_{i(\ell,m)+1}=\sigma$; the walk has hit zero. 
If $L_{i(\ell,m)+1} < \ell$ and $\ell \ge 2$ then in fact $L_{i(\ell,m)+1}=\ell-2$ and $S(\tau_{i(\ell,m)+1}) = 2^{\ell-1}-1$. 
Since $S(\tau_{i(\ell,m+1)}) \ge 2^\ell$, in this case the random walk must complete at least one upcrossing of $[2^{\ell-1},2^\ell)$ between time $\tau_{i(\ell,m)+1}$ and $\tau_{i(\ell,m+1)}$. (Indeed, it must complete exactly one.)

In either case, between time $\tau_{i(\ell,m)}$ and $\tau_{i(\ell,m+1)}$ the random walk either hits zero or else completes an upcrossing of either $[2^{\ell-1},2^{\ell})$ or of $[2^{\ell+1},2^{\ell+2})$. The result follows. 
\end{proof}

\subsection{Occupation times of scales}\label{sec:occ_scale}
We next study  the amount of time the random walk spends at each scale. 
For $0 \le t \le \sigma$, write $\Lambda(t)$ for the current scale at time $t$: formally, with $j = \max\{i: \tau_i \le t\}$, let $\Lambda(t) = L_j$. 
Then, for $x \ge 0$ let $N_\ell(x) = \#\{0 \le t < \min(x,\sigma): \Lambda(t)=\ell\}$. We write $N_\ell=N_\ell(\sigma)$ for the total time spent at scale $\ell$ before hitting $0$. 
Finally, fix $\ell \ge 0$, write $\tau = \min\{t: S_t \not \in [2^{\ell-1},2^{\ell+2})\}$, and let  
\[
n_\ell = \min\left\{t \in \N: \sup_{x \in [2^{\ell-1},2^{\ell+2})} \psub{x}{\tau \ge t} \le 1/2\right\}. 
\]
The first lemma of the section says that the time before the random walk leaves scale $\ell$ is stochastically dominated by $n_{\ell}$ times a geometric random variable. 
\begin{lem}\label{lem:geom1}
For all $k,\ell \ge 0$, with $\tau = \min\{t: S_t \not \in [2^{\ell-1},2^{\ell+2})\}$, 
\[
\sup_{x \in [2^{\ell-1},2^{\ell+2})} \psub{x}{\tau \ge kn_\ell} \le 2^{-k}\, .
\]
\end{lem}
\begin{proof}
Use the Markov property. 
\end{proof}
For the coming lemma, let $(G_i,i \ge 1)$ be a sequence of independent Geometric$(1/2)$ random variables, so 
$\p{G_i= j} = 2^{-(j+1)}$ for integers $j \ge 0$. Also, recall that $\tau_{i(\ell,j)}$ is the $j$'th time the random walk visits scale $\ell$. 
\begin{lem}\label{lem:geom2}
For any integers $k,\ell \ge 0$ and $m \ge 1$, for all $z \in \Z$, 
\[
\psub{z}{\sum_{j=1}^{m} (\tau_{i(\ell,j)+1}-\tau_{i(\ell,j)}) \ge (k+m)n_{\ell}} \le \p{\sum_{j=1}^{m} G_j > k}\, .
\]
\end{lem}
\begin{proof}
For any $1 \le j \le m$, by the strong Markov property and Lemma~\ref{lem:geom1}, for integer $a \ge 0$ we have 
\[
\pcond{\tau_{i(\ell,j)+1}-\tau_{i(\ell,j)}\ge an_\ell}{(S_t,0 \le t \le \tau_{i(\ell,j)})}{} \le 2^{-a} = \p{G_j \ge a}. 
\] 
The sequence $(\lfloor (\tau_{i(\ell,j)+1}-\tau_{i(\ell,j)})/n_{\ell}\rfloor,j \ge 1)$ is thus stochastically dominated by $(G_j,j \ge 1)$. 
Finally, if $\sum_{j=1}^{m} (\tau_{i(\ell,j)+1}-\tau_{i(\ell,j)}) \ge (k+m)n_{\ell}$ then 
$\sum_{j=1}^m \lfloor (\tau_{i(\ell,j)+1}-\tau_{i(\ell,j)})/n_{\ell}\rfloor > k$, and the lemma follows. 
\end{proof}
For $s \ge 0$, let 
\[
H_\ell(s) = \sum_{t < \min(s,\sigma): \Lambda(t)=\ell} \frac{1}{S_t}, 
\]
and write $H_\ell=H_\ell(\sigma)$. Note that $2^{\ell-1} H_\ell(s) \le N_\ell(s) \le 2^{\ell+2} H_\ell(s)$ for all $\ell$ and $s$.
\begin{thm}\label{thm:NL_bd}
For integer $\ell \ge 0$ and real $b > 0$, for all $z \in \Z$, 
\[
\psub{z}{H_\ell \ge b\cdot \frac{n_{\ell}}{2^{\ell-1}}} \le 
\psub{z}{N_\ell \ge b\cdot n_\ell} \le 
\min(1,z/2^{\ell-1}) \cdot 2^{1-b/18}\, .
\]
\end{thm}
\begin{proof}
The first inequality is immediate from the fact that $N_\ell \ge 2^{\ell-1}H_\ell$. For the second, we write
\[
\psub{z}{N_\ell \ge b\cdot n_{\ell}}  = \psub{z}{N_\ell \ne 0} \cdot\psub{z}{N_\ell \ge b\cdot n_{\ell}~|~N_\ell \ne 0}; 
\]
we shall bound the first term by $\min(1,z/2^{\ell-1})$ and the second by $2^{1-b/18}$. 

First, $N_\ell \ne 0$ precisely if the random walk visits scale $\ell$ before hitting zero. When the random walk has scale $\ell$ its position lies in the interval $[2^{\ell-1},2^{\ell})$, so Lemma~\ref{lem:interval} gives that 
$\psub{z}{N_\ell \ne 0} \le \min(1,z/2^{\ell-1})$. 

In bounding the second term, we assume that $b \ge 18$, as otherwise the required bound is trivial. 
We use the strong Markov property at time $\tau = \inf\{t: \Lambda(t)=\ell\}$ to write 
\[
\psub{z}{N_\ell \ge b\cdot n_{\ell}~|~N_\ell \ne 0} \le \sup_{x \in \Z} \psub{x}{N_\ell \ge b\cdot n_{\ell}}. 
\]
We bound the latter by $2^{1-b/18}$ to complete the proof. 

For the remainder of the argument, the following description of $N_\ell$ is more useful. 
Recalling that $M(\ell)$ is the total number of visits to scale $\ell$ before hitting zero, we have 
\[
N_\ell = \sum_{j=1}^{M(\ell)} (\tau_{i(\ell,j)+1}-\tau_{i(\ell,j)}). 
\]
It thus suffices to bound $\sup_{x \in \Z}\psub{x}{\sum_{j=1}^{M(\ell)} (\tau_{i(\ell,j)+1}-\tau_{i(\ell,j)})\ge b\cdot n_{\ell}}$. 

Assume for the moment that $b$ is an integer. 
Fix any positive integers $k$ and $m$ such that $k+m=b$. 
 If $\sum_{j=1}^{M(\ell)} (\tau_{i(\ell,j)+1}-\tau_{i(\ell,j)}) \ge b\cdot n_{\ell}$ then either $M(\ell) > m$ 
or $\sum_{j=1}^{m} (\tau_{i(\ell,j)+1}-\tau_{i(\ell,j)}) \ge (k+m)n_{\ell}$. 
Proposition~\ref{prop:up_bd} states that $M(\ell) \le U(\sigma;[2^{\ell-1},2^{\ell}))+ U(\sigma;[2^{\ell+1},2^{\ell+2}))$, 
and Proposition~\ref{prop:upcrossing} then implies 
\begin{align*}
\p{M(\ell) > m} & \le \p{U(\sigma;[2^{\ell-1},2^{\ell})) \ge (m+1)/2} + \p{U(\sigma;[2^{\ell+1},2^{\ell+2}))\ge (m+1)/2} \\
			& \le \pran{\frac{2^{\ell-1}-1}{2^{\ell}}}^{(m+1)/2}+\pran{\frac{2^{\ell+1}-1}{2^{\ell+2}}}^{(m+1)/2 } \\
			& < \frac{1}{2^{(m-1)/2}}\, .
\end{align*}
Provided that $2m \le k$, a Chernoff bound then gives 
\[
\p{\sum_{j=1}^{m} G_j > k} = \p{\mathrm{Bin}(k,1/2) < m} \le \exp\pran{-\frac{(k-2m)^2}{2k}}\, .
\]
By Lemma~\ref{lem:geom2} we obtain that 
\[
\psub{x}{\sum_{j=1}^{M(\ell)} (\tau_{i(\ell,j)+1}-\tau_{i(\ell,j)}) \ge b\cdot n_{\ell}} \le \frac{1}{2^{(m-1)/2}} + \exp\pran{-\frac{(k-2m)^2}{2k}}\, .
\]
Choose $m \in [b/9+1,b/9+2]$; using that $b \ge 18$, straightforward arithmetic shows that the sum on the right is then  
bounded by $2^{-b/18} + e^{-b/18} \le 2\cdot2^{-b/18}$. 
This completes the proof when $b$ is integer. For general $b$, the same argument yields the bound 
$2^{-\lfloor b \rfloor/18}+e^{-\lfloor b \rfloor/18}$; but this is still less than $2 \cdot 2^{-b/18}$. 
\end{proof}
The next corollary provides a cleaner probability tail bound for $H(s)$. For a sequence $\rb=(b_\ell,0 \le \ell \le m)$ of positive real numbers, let
\[V(\rb) = 36 \sum_{0 \le \ell \le m} \frac{b_\ell n_\ell}{2^\ell}\, \quad \text{and}\quad
\Delta(\rb) = 4\sum_{0 \le \ell \le m} 2^{-\ell-b_\ell}
\, .
\] 
\begin{cor}\label{cor:genericbound}
Fix a positive integer $m$ and positive real numbers $\rb=(b_m,0 \le \ell \le m)$. Then for any positive integer $s$, 
\[
\psub{1}{ H(s) > V(\rb) + \frac{s}{2^{m-1}}} \le \Delta(\rb).
\]
\end{cor}
\begin{proof}
 By Theorem~\ref{thm:NL_bd}, we have that for any $b > 0$, 
\[
\psub{1}{H_\ell(\sigma) \ge 36 b \cdot \frac{n_\ell}{2^{\ell}}} \le 4 \cdot 2^{-\ell-b}. 
\]
Now fix $s$, $m$, and $\rb=(b_\ell,0 \le \ell \le m)$ as in the statement of the corollary. Then 
\[
\sum_{\ell > m} H_\ell(s) \le \frac{s}{2^{m-1}}\, , 
\]
so $H(s) \le \sum_{\ell \le m} H_{\ell}(s) +s/2^{m-1}$, and therefore 
\[
\psub{1}{ H(s) > V(\rb) + \frac{s}{2^{m-1}}} \le 
\sum_{\ell \le m} 
\psub{1}{H_\ell(s) > 36\frac{b_\ell n_\ell}{2^\ell}} 
\le \Delta(\rb)\, .  
\]
\end{proof}
In the next section we use Corollary~\ref{cor:genericbound} to derive upper bounds for probabilities of the form $\psub{1}{H(\sigma) \ge x f(\sigma)}$, under a range of assumptions on the step size distribution. The intuition behind all these results is that if $s \le \sigma$ then the dominant 
contribution to the sum $H(s)$ should typically come from the largest scale reached by the walk before hitting zero. In other words, if $s \le \sigma$ then, writing $g(s) = \sup\{\ell: n_{\ell} \le s\}$, we expect that $H(s)$ is typically around $s/2^{g(s)}$. 

\section{Relating $\sigma$ and $H(\sigma)$}\label{key_results}
\begin{thm}\label{thm:hvol_finvar}
There exists an absolute constant $C^*$ such that the following holds. 
Fix an integer random variable $X$ with $\p{X \ge -1}=1$ and $\e{X} \le 0$, and let $(S_t,t \ge 0)$ be a random walk with step distribution $X$. Then writing $p_0=\p{X=0}$, for all $x \ge 1$, 
\[
\psub{1}{H(\sigma) \ge C^* x \cdot \frac{\sigma^{1/2}}{(1-p_0)^{1/2}}} \le e^{-x^2}\, .
\]
\end{thm}
\begin{proof}
Fix $c > 0$ and $m > 0$. For $0 \le \ell \le m$ let $b_\ell = (1-p_0)4^c \cdot 2^{(m-\ell)/2}$, and write $\rb=(b_\ell,0 \le \ell \le m)$. 
By Lemma~\ref{lem:nl_bound} we have $n_{\ell} \le C \cdot 4^{\ell}/(1-p_0)$, so 
\[
V(\rb) = 36 \sum_{0 \le \ell \le m} \frac{b_\ell n_\ell}{2^{\ell} }
= 36C 4^c \sum_{0 \le \ell \le m} 2^\ell 2^{(m-\ell)/2} < 
C'4^c 2^m\, ,
\]
for some absolute constant $C'> 1$. 

It is straightforward to check that for all $x \ge 1$, $\sum_{\ell \ge 0} 2^{\ell-x2^{\ell/2}} < 6\cdot 2^{-x}$. Thus, provided $4^c(1-p_0) \ge 1$, we also have 
\begin{align*}
\Delta(\rb)	& = 4 \sum_{0 \le \ell \le m} 2^{-\ell-b_\ell}
		 = \frac{4}{2^m} \sum_{0 \le \ell \le m} 
		2^{m-\ell - (1-p_0)4^c 2^{(m-\ell)/2}}\\
		& = \frac{4}{2^m} \sum_{0 \le \ell \le m}
		2^{\ell-4^c(1-p_0)2^{\ell/2}} < \frac{24}{2^m} \cdot 2^{-4^c(1-p_0)}\,.
\end{align*}
 It follows by Corollary~\ref{cor:genericbound} that if $4^c(1-p_0) \ge 1$ then for any $s>0$, 
\[
\psub{1}{H(s) \ge C' 4^c 2^m + \frac{s}{2^{m-1}}} \le \frac{24}{2^m} \cdot 2^{-4^c(1-p_0)}\, .
\]
Taking $s = 4^{m+c}$ gives 
$C'4^c2^m+s/2^m < (C'+2) 4^c 2^m = (C'+2) 2^c s^{1/2}$. On the event that $\sigma \in (4^{m+c-1},4^{m+c}]$, we have $H(s) = H(\sigma)$ and $2\sigma^{1/2} \ge s^{1/2}$, so this yields 
\[
\psub{1}{H(\sigma) \ge (C'+2) 2^{c+1} \sigma^{1/2}, \sigma \in (4^{m+c-1},4^{m+c}]} \le 
\frac{24}{2^m} \cdot 2^{-4^c(1-p_0)}\, .
\]
Moreover, if $\sigma \le 4^{c}$ then 
$2^{c+1}\sigma^{1/2} \ge \sigma$ so 
\[
\psub{1}{H(\sigma) \ge (C'+2) 2^{c+1} \sigma^{1/2}, \sigma \le 4^{c}} = 0\, .
\]
Using the two preceding bounds, we obtain that 
whenever $4^c(1-p_0) \ge 1$, 
\begin{align*}
\psub{1}{H(\sigma) \ge (C'+2) 2^{c+1} \sigma^{1/2}}
& = \sum_{m \ge 1} 
\psub{1}{H(\sigma) \ge (C'+2) 2^{c+1} \sigma^{1/2}, \sigma \in (4^{m+c-1},4^{m+c}]} \\
& \le 24 \cdot 2^{-4^c(1-p_0)}\, .
\end{align*}
Taking $x=2^c(1-p_0)^{1/2}\ge 1$, the above bound becomes 
\[
\psub{1}{H(\sigma) \ge 2(C'+2) x \cdot \frac{\sigma^{1/2}}{(1-p_0)^{1/2}}} \le 24 \cdot 2^{-x^2}\, ,
\]
from which the theorem follows easily. 
\end{proof} 

\begin{thm}
\label{thm:var_precise}
There exists an absolute constant $C$ such that the following holds. Fix an integer random variable $X$ with $\p{X \ge -1}=1$ and with $\e{X} \le 0$ and $\E{(X-\e{X})^2} = v \in (0,\infty)$, and let $(S_t,t \ge 0)$ be a random walk with step distribution $X$. 
Then there exists $x_0 > 0$ depending only on the law of $X$ such that for all $x \ge x_0$ and all $s \ge 1$, 
\[
\psub{1}{H(\sigma) \ge x\sigma^{1/2},\sigma \ge s} 
\le \frac{Cx}{s^{1/2}} e^{-vx^2/C}\, .
\]
\end{thm}
\begin{proof}
Observe that 
\begin{align*}
\E{(X_1-X_2)^2} & = \E{\E{(X_1-X_2)^2|X_2}} \\
& \ge \E{\E{(X_1-\e X_1)^2|X_2}} = \E{(X-\e X)^2} = v\, .
\end{align*}
By monotone convergence, we may thus choose $k_0>0$ such that 
\[
\E{(X_1-X_2)^2\I{|X_1-X_2| \le 2^k}} \ge v/4
\]
for all $k \ge k_0$.
The first bound of Theorem~\ref{thm:kesten} yields that for all $k \ge k_0$ we have $Q(S_n,2^k) \le 2^{k+1}/(v n)^{1/2}$.
Writing $\tau_k = \inf\{t: S_t \not \in [2^{k-1},2^{k+2})\}$, it follows that for all $x \in [2^{k-1},2^{k+2})$, if $n \ge 4^{k+4}/v$ then 
\[
\psub{x}{\tau_k \ge n} \le \psub{x}{S_n \in [2^{k-1},2^{k+2})} \le Q(S_n,2^{k+2}) \le \frac{2^{k+3}}{(vn)^{1/2}} \le \frac{1}{2}\, .
\]
We thus have $n_k \le 4^{k+4}/v$ for all $k \ge k_0$. For smaller $k$ we use the bound $n_k \le C \cdot 4^{k}/(1-p_0)$ of Lemma~\ref{lem:nl_bound}. 

Now fix $B\ge 1$ and $m \ge k_0$, and for $0 \le \ell \le m$ let $b_\ell = 2^{(m-\ell)/2}B$. 
Then with $\rb = (b_{\ell},0 \le \ell \le m)$, in the notation of Corollary~\ref{cor:genericbound} we have 
\begin{align*}
V(\rb) & = 36 B\sum_{0 \le \ell \le m} \frac{b_\ell n_\ell}{2^\ell} \\
	& \le 36 B \bigg(\frac{C}{1-p_0}\sum_{0 \le \ell < k_0} 2^{(m+\ell)/2}+ \frac{4^4}{v} \sum_{k_0 \le \ell \le m} 2^{(m+\ell)/2}\bigg). \\
	& \le C'B \left(\frac{2^{(m+k_0)/2}}{(1-p_0)} + \frac{2^m}{v} \right)\, .
\end{align*}
Provided $m \ge k_0 + 2\log(v/(1-p_0)))$, the final quantity is at most $C'B2^{m+1}/v$. Moreover, by reprising the bound on $\Delta(\rb)$ from Theorem~\ref{thm:hvol_finvar}, we obtain that $\Delta(\rb) \le 24/2^{m+B}$. 
Letting $m_0 = k_0 + 2\log(v/(1-p_0))$, it follows by Corollary~\ref{cor:genericbound} that for $m \ge m_0$, for all $t > 0$ we have 
\begin{equation}\label{eq:intermediate}
\psub{1}{H(s) \ge C'B\frac{2^{m+1}}{v} + \frac{t}{2^{m-1}}} \le \frac{24}{2^{m+B}}\, .
\end{equation}

Next, fix $x \ge 4^{m_0}$, and for $i \ge 0$ let 
$s_i = x^2 \cdot 4^i$. 
Then for any $j \ge 0$ we may write 
\begin{align*}
\psub{1}{H(\sigma) \ge x\sigma^{1/2},\sigma \ge s_j} & = \sum_{i \ge j} \psub{1}{H(\sigma) \ge x\sigma^{1/2}, \sigma \in [s_i,s_{i+1})} \\
	& \le \sum_{i \ge j} \psub{1}{H(s_{i+1}) \ge x s_i^{1/2}}\, .
\end{align*}
To bound the final summands, take $t=s_i$ and $m = i$, and $B=vx^2/C'$ in (\ref{eq:intermediate}). Then $t/2^{m-1}=2x t^{1/2}$ and 
\[
C'B \frac{2^{m+1}}{v} = x^2 \cdot 2^{m+1} = 2xt^{1/2}\, ,
\]
and $2^{m} = t^{1/2}/x$, so (\ref{eq:intermediate}) implies that  
\[
\psub{1}{H(s_i) \ge 4xs_i^{1/2}} \le \frac{24x}{s_i^{1/2}}2^{-vx^2/C'}\, .
\]
and summing over $i \ge j$ yields that 
\[
\psub{1}{H(\sigma) \ge 4x\sigma^{1/2},\sigma \ge s_j} \le \frac{48x}{s_j^{1/2}} \cdot  2^{-vx^2/C'}\,, . 
\]
For any $s \ge 4^{m_0}$ there is $j \ge 0$ such that $s \in [s_j,s_{j+1})$. For this value of $j$ we also have $2^{-(m_0+j)} \le xs^{-1/2}$, so using the preceding bound we obtain 
\begin{align*}
\psub{1}{H(\sigma) \ge 4x\sigma^{1/2},\sigma \ge s}& \le 
\psub{1}{H(\sigma) \ge 4x\sigma^{1/2},\sigma \ge s_j} \\
& \le \frac{96x}{s^{1/2}} 2^{-vx^2/C'}\, .
\end{align*}
Finally, if $\sigma < x^2$ then $H(\sigma) \le \sigma < x\sigma^{1/2}$, so 
\[
\psub{1}{H(\sigma) \ge 4x\sigma^{1/2}} = 
\psub{1}{H(\sigma) \ge 4x\sigma^{1/2}, 
\sigma \ge x^2} \le 96 \cdot 2^{-vx^2/C'}\, . \qedhere
\] 
\end{proof}

\begin{thm}\label{thm:stable_attempt}
There exists an absolute constant $C^*$ such that the following holds. Fix $\alpha \in (1,2]$ and write $M=M(\alpha) = \sup_{\ell \ge 0} n_\ell/2^{\alpha \ell} \le \infty$.  Then for any real $v \ge 0$ we have 
\[
\psub{1}{H(\sigma) > \frac{C^*M}{(\alpha-1)} \cdot v \cdot \sigma^{(\alpha-1)/\alpha}}
 \le e^{-v^{\alpha}}\, .
\]
\end{thm}
\begin{proof}
Fix $c > 0$ and an integer $m > 0$. For $0 \le \ell \le m$ let $b_{\ell} = x 2^{(m-\ell)(\alpha-1)/2}$, and write $\rb = (b_{\ell},0 \le \ell \le m)$. By assumption we have $n_{\ell} \le M 2^{\alpha \ell}$, so 
\[
V(\rb) = 
 \sum_{0 \le \ell \le m} \frac{b_\ell n_\ell}{2^{\ell} }
\le 
36M x \sum_{0 \le \ell \le m} 2^{(\alpha-1)\ell} 2^{(m-\ell)(\alpha-1)/2} < 
C'Mx 2^{(\alpha-1)m}\, ,
\]
where $C'>0$ depends only on $\alpha$. 
It is straightforward to check that for if $x > 6/(\alpha-1)$ then 
$\sum_{\ell \ge 0} 2^{\ell-x \cdot 2^{\ell(\alpha-1)/2}} \le 2^{1-x}$. 
Thus, provided $x > 6/(\alpha-1)$, we also have 
\begin{align*}
\Delta(\rb)	& = 4 \sum_{0 \le \ell \le m} 2^{-\ell-b_\ell}
		 = \frac{4}{2^m} \sum_{0 \le \ell \le m} 
		2^{m-\ell - x 2^{(m-\ell)(1-\alpha)/2}}\\
		& = \frac{4}{2^m} \sum_{0 \le \ell \le m}
		2^{\ell-x2^{\ell(1-\alpha)/2}} < \frac{1}{2^{x+m-3}}\, .
\end{align*}
It follows from Corollary~\ref{cor:genericbound} that for any integer $s>0$ we have
\[
\psub{1}{H(s) \ge C'Mx2^{(\alpha-1)m}+ \frac{s}{2^m}} \le 
\frac{1}{2^{x+m-3}}. 
\]
Taking $s=x\cdot 2^{\alpha m}$ gives 
$C'Mx2^{(\alpha-1)m}+ \frac{s}{2^m} = (C'+1)Mx2^{(\alpha-1)m}$. On the event that $\sigma \in [x2^{\alpha (m-1)},x2^{\alpha m}]$ 
we have $H(s)=H(\sigma)$ and $2x^{1/\alpha}\sigma^{(\alpha-1)/\alpha} \ge 2x 2^{(\alpha-1)(m-1)} \ge x2^{(\alpha-1)m}$, 
so taking $C''=2(C'+1)$ the above bound yields
\[
\psub{1}{H(\sigma) \ge C''Mx^{1/\alpha} \sigma^{(\alpha-1)/\alpha}, \sigma \in [x2^{\alpha (m-1)},x2^{\alpha m}]} 
\le \frac{1}{2^{x+m-3}}\, .
\]
Moreover, if $\sigma > x$ then $x^{1/\alpha} \sigma^{(\alpha-1)/\alpha} < \sigma$. We always have $M \ge n_1/4 \ge 1/4$, and we may assume $C'' > 4$, so it follows that 
\[
\p{H(\sigma) \ge C''Mx^{1/\alpha}\sigma^{(\alpha-1)/\alpha}, \sigma < x} \le 
\p{H(\sigma) \ge x^{1/\alpha}\sigma^{(\alpha-1)/\alpha}, \sigma < x} = 0. 
\]
Using the two preceding bounds, we obtain that whenever $x > 6/(\alpha-1)$ we have 
\begin{align*}
& \psub{1}{H(\sigma) \ge C''Mx^{1/\alpha} \sigma^{(\alpha-1)/\alpha}}\\
 & = \sum_{m \ge 1} 
\psub{1}{H(\sigma) \ge C''Mx^{1/\alpha} \sigma^{(\alpha-1)/\alpha}, \sigma \in [x2^{\alpha (m-1)},x2^{\alpha m}]} \\
& \le \frac{8}{2^x},
\end{align*}
from which the result follows easily. 
\end{proof}

\begin{prop}\label{prop:inf_var_generic}
There exists an absolute constant $C^*$ such that the following holds. 
Fix an integer random variable $X$ with $\p{X \ge -1}=1$, $\e{X} \le 0$, and $\E{X^2} = \infty$. Let $(S_t,t \ge 0)$ be a random walk with step distribution $X$. Then for all $\delta > 0$ there exists $m_0=m_0(\delta)$ such that for any integers $m \ge m_0$ and $s > 0$, and any real $B \ge 1$, 
\[
\psub{1}{H(s) \ge \delta B \cdot 2^m + \frac{s}{2^{m-1}}} \le \frac{1}{2^{m+B}}. \, .
\]
\end{prop}
\begin{proof}
In this proof we write $I_k = [2^{k-1},2^{k+2})$ and $\tau_k = \inf\{t: S_t \not \in I_k\}$. 
Fix $\eps > 0$ small, and let $k_0$ be large enough that 
$\E{X^2 \I{X \in [0,2^{k_0})}} \ge C/(\eps \p{X=-1})$, where $C$ is the constant from Corollary~\ref{cor:disperse}. Then for all $k \ge k_0$ we have $Q(S_n,2^k) \le \eps 2^k/n^{1/2}$, so for all $x \in I_{k}$, if $n \ge \eps^2 4^{k+3}$ then 
\[
\psub{x}{\tau_k \ge n} \le \psub{x}{S_n \in I_k} \le Q(S_n,2^{k+2}) \le \frac{\eps 2^{k+2}}{n^{1/2}} \le \frac{1}{2}\, .
\]
It follows that $n_k \le \eps^2 4^{k+3}$ for all $k \ge k_0$. For smaller $k$ we use the bound $n_k \le C \cdot 4^{k}/(1-p_0)$ of Lemma~\ref{lem:nl_bound}. 

Now fix $B\ge 1$ and $m \ge k_0$, and for $0 \le \ell \le m$ let $b_\ell = 2^{(m-\ell)/2}B$. 
Then with $\rb = (b_{\ell},0 \le \ell \le m)$, in the notation of Corollary~\ref{cor:genericbound} we have 
\begin{align*}
V(\rb) & = 36 \sum_{0 \le \ell \le m} \frac{b_\ell n_\ell}{2^\ell} \\
	& \le C B \bigg(\frac{1}{1-p_0}\sum_{0 \le \ell < k_0} 2^{(m+\ell)/2}+ 4^c\eps^2 \sum_{k_0 \le \ell \le m} 2^{(m+\ell)/2}\bigg). \\
	& \le C'B \left(\frac{2^{(m+k_0)/2}}{(1-p_0)} + \eps^2 2^m \right)\, ,
\end{align*}
where $C$ and $C'$ are absolute constants. 
Provided $m \ge k_0 + 2\log(1/(\eps^2(1-p_0)))$, the final quantity is at most $2C'B\eps^2\cdot 2^{m}$. Moreover, by reprising the bound on $\Delta(\rb)$ from Theorem~\ref{thm:hvol_finvar}, we obtain that $\Delta(\rb) \le 24/2^{m+B}$. 

Letting $m_0 = k_0 + 2\log(1/(\eps^2(1-p_0)))$, it follows by Corollary~\ref{cor:genericbound} that for $m \ge m_0$, for all $s > 0$ we have 
\[
\psub{1}{H(s) \ge 2C'B\eps^2 2^m + \frac{s}{2^{m-1}}} \le \frac{24}{2^{m+B}}\, .
\]
Taking $\eps$ small enough that $2C'\eps^2 < \delta/5$, say, the result then follows. 
\qedhere
\end{proof}

\begin{cor}\label{cor:inf_var_generic}
Under the conditions of Proposition~\ref{prop:inf_var_generic}, for any $\eps > 0$ there exists $n_0$ such that for all $x > 0$ and all $s \ge x^2n_0$, 
\[
\p{H(\sigma) \ge x\sigma^{1/2},\sigma \ge s} \le \frac{x}{s^{1/2}}e^{-x^2/\eps}. 
\]
\end{cor}
\begin{proof}
Fix fix $\eps > 0$, let $\delta=\eps/2$, and let $m_0=m_0(\delta)$ be as in Proposition~\ref{prop:inf_var_generic}. 
Now fix $x > 0$. We consider the cases $x\ge (2\eps)^{1/2}$ and $x < (2\eps)^{1/2}$ separately. 

First suppose $x\ge (2\eps)^{1/2}$. For $i \ge 0$ 
let $s_i = x^2 \cdot 4^{m_0+i}$.  
For any $j \ge 0$ and $y \ge x$ we may write
\begin{align*}
\psub{1}{H(\sigma) \ge y\sigma^{1/2},\sigma \ge s_j} & = \sum_{i \ge j} \psub{1}{H(\sigma) \ge y\sigma^{1/2}, \sigma \in [s_i,s_{i+1})} \\
	& \le \sum_{i \ge j} \psub{1}{H(s_{i+1}) \ge y s_i^{1/2}}\\
	& = \sum_{i \ge j} \psub{1}{H(s_{i+1}) \ge ys_{i+1}^{1/2}/2}\, .
\end{align*}
Taking $s=s_{i+1}$ and $m = m_0+i$, we have 
$s/2^{m-1} = 2xs^{1/2}$; setting 
$B= x^2/\delta$, we also obtain 
$\delta B\cdot 2^m = x^2 2^m = xs^{1/2}$. 
Proposition~\ref{prop:inf_var_generic} then implies that 
\[
\psub{1}{H(s_{i+1}) \ge 3 xs_{i+1}^{1/2}/2} \le 2^{-m_0-i-B}.
\]
Now take $y=3x$, sum over $i \ge j$, and use that $s_i^{1/2}/x=2^{m_0+i}$; this yields that 
\[
\psub{1}{H(\sigma) \ge 3x\sigma^{1/2},\sigma \ge s_j} \le 2^{1-m_0-j-B} = \frac{2x}{s_j^{1/2}} e^{-x^2/\delta}
\]

For any $s \ge s_0$ there is $j \ge 0$ such that $s \in [s_j,s_{j+1})$. For this value of $j$ we have $x/s^{1/2} < 2x/s_j^{1/2}$, so the using preceding bound we obtain 
\[
\psub{1}{H(\sigma) \ge 3x\sigma^{1/2},\sigma \ge s}\le 
\psub{1}{H(\sigma) \ge 3x\sigma^{1/2},\sigma \ge s_j} 
\le
\frac{4x}{s^{1/2}} e^{-x^2/\delta}\, .
\]
To conclude, recall that $\eps=2\delta$, so 
$e^{-x^2/\delta}=e^{-2x^2/\eps}$. 
Since $x \ge 2\delta^{1/2}$ we have 
$e^{-2x^2/\eps} \le e^{-x^2/\eps}/4$, and the result follows in this case.

Now suppose $0 < x < (2\eps)^{1/2}$.
Since $\E{X^2}=\infty$, 
by Theorem~4.1 of \cite{MR0231419} 
we have $\sup_{z \in \Z} \p{S_n=z} = o(n^{-1/2})$. 
We may thus choose $s_0$ large enough that for all $n \ge s_0$, 
$\sup_{z \in \Z} \p{S_n=z} \le x n^{-1/2}/4$. 
By the cycle lemma, for all $n$ we have 
\[
\psub{1}{\sigma=n} = \frac{1}{n} \p{S_n=n-1},
\]
which for $s \ge s_0$ yields 
\[
\psub{1}{H(\sigma) \ge x \sigma^{1/2},\sigma \ge s}
\le \sum_{n \ge s}
\psub{1}{\sigma=n} \le \frac{x}4 \sum_{n \ge s} n^{-3/2} 
\le \frac{x}{8 (s-1)^{1/2}}. 
\]
For $x < (2\eps)^{1/2}$ we have $e^{-x^2/\eps} > e^{-2} > 1/8$, so for $s$ sufficiently large the preceding bound is at most $(x/s^{1/2})e^{-x^2/\eps}$. 
\end{proof}

\section{Proofs of the main theorems}\label{sec:proofs}
\begin{proof}[Proof of Theorem~\ref{thm:general}]
First, note that if $\mu$ is supercritical then conditionally given that $|T|=\infty$, almost surely 
$\h(T)=\w(T)=\infty$. On the other hand, writing $\hat{\mu}$ for the measure with $\hat{\mu}(i) = \mu(i)q^{i-1}$, where $q=\p{|T|<\infty}$, then 
given that $|T|<\infty$, the conditional law of $|T|$ is 
$\mathrm{GW}(\hat{\mu})$, and $\hat{\mu}$ is subcritical or critical. Since the bounds of the theorem only depend on $\mu$ through $\mu(1)$, and $\hat{\mu}(1)=\mu(1)$, it thus suffices to prove the theorem for critcal and subcritical trees.

In light of the preceding paragraph, the second bound is now immediate from Propositions~\ref{bfq_equiv} and~\ref{prop:ht_bd} and Theorem~\ref{thm:hvol_finvar}. For the first, we use that for any tree $T$, $\h(T)\w(T) \ge |T|$. 
We then have
\begin{align*}
\p{\h(T) \ge (Cx)^2 \frac{\w(T)}{1-\mu(1)}} & \le 
\p{\h(T)^2 \ge (Cx)^2 \frac{|T|}{1-\mu(1)}} \\
& = \p{\h(T) \ge Cx \frac{|T|^{1/2}}{(1-\mu(1))^{1/2}}}\, ,
\end{align*}
which yields the first bound. 
\end{proof}
\begin{proof}[Proof of Theorem~\ref{thm:fixed_var}]
Propositions~\ref{bfq_equiv} and~\ref{prop:ht_bd} and Theorem~\ref{thm:var_precise} together 
imply that there exists $x_0$ depending only on $\mu$ such that for all $x \ge x_0$ and all $n \ge 1$, 
\[
\p{\h(T) \ge x|T|^{1/2}, |T| \ge n} \le \frac{Cx}{n^{1/2}}e^{-vx^2/C}\, ,
\]
where $C>0$ is a universal constant. For such $x$ we then have 
\[
\p{\h(T) \ge (Cx)|T|^{1/2}} = 
\p{\h(T) \ge (Cx)|T|^{1/2}, |T| \ge (Cx)^2} 
\le 
e^{-Cvx^2}.
\]
This proves the second probability bound, and the first bound follows by the same argument used in proving Theorem~\ref{thm:general}. 

Now suppose $\mu$ is critical, and let $(S(v_{i+1}),0 \le i < |T|)$ be the breadth-first queue process of $T$.  
By Proposition~\ref{bfq_equiv}, this process has the same law as $(S_i,0 \le i < \sigma)$, where $(S_i,i \ge 0)$ is a random walk with $S_0=1$ and jump distribution $\nu$ defined by $\nu(i)=\mu(i+1)$, and $\sigma=\inf\{n: S_n = 0\}$. Since $\mu$ is critical, $(S_i,i \ge 0)$ is centered, so by the local central limit theorem there is $n_0$ such that for all $n \ge n_0$,   
$\psub{1}{S_n=0} \ge (2vn)^{-1/2}$. By the cycle lemma, we then have 
\begin{align*}
\p{|T|\ge n} & = \psub{1}{\sigma\ge n}
\\
&= \sum_{m \ge n} \psub{1}{\sigma=m} \\
&= \sum_{m \ge n} \frac{1}{m}\psub{1}{S_m=0}\\
&\ge \sum_{m \ge n} \frac{1}{(2v)^{1/2}m^{3/2}} \\
& \ge \frac{1}{(8v)^{1/2} n^{1/2}}\, .
\end{align*}

Combined with the above probability bound, we obtain that for $n\ge n_0$, 
\[
\p{\h(T) \ge x|T|^{1/2}~|~|T| \ge n} \le 
C'v^{1/2}x e^{-vx^2/C},
\]
for some absolute constant $C'$. For $n < n_0$ 
we have $\p{|T| \ge n} \ge (8vn_0)^{-1/2}$, so  
\[
\p{\h(T) \ge x|T|^{1/2}~|~|T| \ge n} \le 
(8vn_0/n)^{1/2} Cx e^{-vx^2/C}\, .
\]
For $x$ sufficiently large we have
\[
C'v^{1/2}x e^{-v(Cx)^2/C}
\le e^{-vx^2}\,\,\mbox{and}\,\, 
(8vn_0)^{1/2}Cxe^{-v(Cx)^2/C} \le e^{-vx^2}\, ,
\]
so the final bound follows. 
\end{proof}
\begin{proof}[Proof of Theorem~\ref{thm:stable}]
The second bound is immediate from Propositions~\ref{bfq_equiv} and~\ref{prop:ht_bd} and Theorem~\ref{thm:stable_attempt}. For the first, arguing as in 
 Theorem~\ref{thm:general}, we have 
 \begin{align*}
\p{\h(T) \ge  x \left(\frac{C M}{\alpha-1}\right)^{\alpha} \w(T)^{\alpha-1}}
& \le \p{\h(T)^{\alpha} \ge x \left(\frac{C M}{\alpha-1}\right)^{\alpha} |T|^{\alpha-1}}\\
& 
= 
\p{\h(T) \ge x^{1/\alpha} \frac{C M}{\alpha-1} |T|^{(\alpha-1)/\alpha}} \\
& \le e^{-x}\, ,
\end{align*}
which is the first bound. 
\end{proof}
\begin{proof}[Proof of Theorem~\ref{thm:inf_var}]
The second bound is immediate from Propositions~\ref{bfq_equiv} and~\ref{prop:ht_bd} and Corollary~\ref{cor:inf_var_generic}. We deduce the first as in the other theorems: since $\h(T)\w(T) \ge |T|$, 
\[
\p{\h(T) \ge x\w(T),\sigma \ge s} 
 \le \p{\h(T)^2 \ge x|T|,\sigma \ge s},
\]
from which the first bound follows. 
\end{proof}

\section{Conclusion}\label{sec:conc}
This section provides a few pointers to related work, and suggests open questions related to the results presented above, as well as some potential strengthenings of said results. 
\begin{enumerate}
\item 
Write $\GW_n(\mu)$ for the law of a Galton-Watson tree with offspring distribution $\mu$ conditioned to have size exactly $n$. A natural question is whether the above theorems can be shown to hold for $T_n \sim \GW_n(\mu)$. In any case where this is possible, it yields a stronger result, as the corresponding theorem for $T\sim \GW(\mu)$ can then be obtained by a suitable averaging over $n$. 
\begin{itemize}
\item[(a)] For Theorem~\ref{thm:general}, such a generalization is false; if $\mu(0)+\mu(1)=1$ then $T_n$ is almost surely a path of length $n$. 
An extension to conditioned Galton-Watson trees may still be possible, but its bounds must include some dependence on both $\mu(0)$ and $\mu(1)$. 

\item[(b)] As mentioned in the introduction, when $\mu$ is critical with finite variance, sub-Gaussian tail bounds for $\h(T_n)/n^{1/2}$ were proved in \cite{addario-berry2013}. It should be possible to strengthen the bounds of \cite{addario-berry2013} to exhibit the same dependence on the variance $v$ as in Theorem~\ref{thm:fixed_var}; such a strengthening would then yield Theorem~\ref{thm:fixed_var} as a corollary. 

\item[(c)] The work \cite{kortchemski17} provides a version of Theorem~\ref{thm:stable} which applies to conditioned Galton-Watson trees. However, this result requires that $\mu$ is in the domain of attraction of a stable law; it should be possible to weaken this requirement, insisting only on upper bounds for the tail probabilities of the offspring distribution. 

\item[(d)] 
As for an analogue of Theorem~\ref{thm:inf_var}, it should hold that 
for all $\eps$ there exists $n_0=n_0(\eps)$ such that for $n \ge n_0$, 
\[
\p{\h(T_n) > xn^{1/2}} \le e^{-x^2/\eps}\, \, \mbox{and} 
\p{\h(T_n) > x\w(T_n)} e^{-x/\eps}. 
\]
These bounds would imply that $\h(T_n)/n^{1/2} \to 0$ in probability and that $\w(T_n)/n^{1/2} \to \infty$ in probability; both limits were conjectured by Janson (see \cite{MR2908619}, Conjectures~21.5 and~21.6)). 
\end{itemize}
\item An intuition which motivated the development of this paper is that for Galton-Watson trees $T$ we expect that $\w(T)\cdot \h(T)$ is typically of order $|T|$. This is true, for example, for conditioned Galton-Watson trees in the domain of attraction of a stable tree. However, in general this need not hold. For example, suppose that $\mu(0)=1/2$, $\mu(1)=1/4$  and $\mu(2^{2^{2^i}})=1/2^i$. Taking $n=2^{2^{2^i}}$, the tree $T_{2n}$ typically contains exactly one node of degree $n$; it has width of order $n$ and height of order $\log n/\log\log n$. Are there examples where of offspring distributions for which $\w(T_n)\h(T_n)/n$ is much larger than $\log n/\log\log n$ with non-vanishing probability?
\item More generally, the range of possible joint behavior of $\h(T_n)$ and $\w(T_n)$ for supercritical conditioned Galton-Watson trees is unclear, and deserves investigation. 
\end{enumerate}
It is not impossible that an extension of the techniques of the current paper could be used to tackle some of the above questions. Our approach essentially requires bounds on the amount of time the random walk associated to the breadth-first queue process spends at small scales. Thus, implementing this for conditioned Galton-Watson trees would require, in particular, universal bounds on how much time a random walk conditioned to first visit $0$ at time $n$ is likely to spend at a given scale. 

\section*{Acknowledgements}
Thank you, Igor Kortchemski and Yuval Peres, for useful discussions during the preparation of this paper. 

This research was funded in part by an NSERC Discovery Grant. Part of the research was carried out at the Isaac Newton Institute for Mathematical Sciences during the programme ``Random Geometry'', supported by EPSRC Grant Number EP/K032208/1. My interest in this problem was sparked during a workshop at McGill's Bellairs research institute in Holetown, Barbados. Thanks are due to all the above institutions and agencies for their support. 

\bibliographystyle{plainnat}

\end{document}